%% file: bezout_fr.tex
\pdfoutput=1
\documentclass{article}
\usepackage{standalone}
\usepackage[utf8]{inputenc}
\usepackage[fleqn]{amsmath}
\usepackage{amsthm}
\usepackage{amssymb}
\usepackage{amsfonts}
\usepackage{wrapfig}
\usepackage{graphicx}
\usepackage[francais]{babel}
\usepackage{hyperref}
\usepackage[ruled]{algorithm}
\usepackage{algorithmicx}
\usepackage{algpseudocode}

\usepackage[margin=2cm]{geometry}
\usepackage{adjustbox}
\usepackage{multirow}

\theoremstyle{plain}

\newtheorem{prop}{Proposition}

\theoremstyle{definition}
\newtheorem{defn}{Définition}
\newtheorem{exmp}{Example}
\theoremstyle{remark}
\newtheorem*{rem}{Remarque}

\newcommand{\N}{\mathbb{N}}

\newcommand{\C}{\mathbb{C}}

\title{Résolution des systèmes polynomiaux: \\un solveur basé sur les matrices de Bezout.}
\author{Jean-Paul Cardinal\\
\texttt{cardinal@math.univ-paris13.fr}}

\begin{document}

\maketitle

\begin{abstract}
Nous proposons un algorithme de calcul numérique des racines d'un système polynomial en intersection complète. Cet algorithme utilise les matrices de Bézout et ne fait appel qu'à des procédures d'algèbre linéaire. Il est possible d'éxécuter l'ensemble des calculs en arithmétique flottante. Une implémentation en Numpy/Octave/Sage est publiée sur le site \cite{jp_code}.
\end{abstract}

\section{Cas d'une variable}
\label{univariable}
Rappelons quelques faits connus sur les polynômes à une variable.\\
Dans toute cette partie nous considérons un polynôme $f = a_0x^d + \dots + a_{d-1}x + a_d$ à coefficients dans $\C$. Notons $\langle f \rangle$ l'idéal engendré par $f$ dans l'anneau de polynômes $\C[x]$ et $A = \C[x]/\langle f \rangle$  son algèbre quotient. Dorénavant $x$ désignera indifféremment la variable $x$, sa projection sur le quotient $A$ ou l'endomorphisme de multiplication par $x$ dans $A$. Une base du $\C$-espace vectoriel $A$ est la {\bf base des monômes} $\bold{x} = (1, x,\cdots, x^{d-1})$.

\subsection{Matrices des opérateurs de multiplication dans $A$}
L'opérateur de multiplication
$x : \left\vert
\begin{array}{c}
A \mapsto A \\
h \mapsto xh
\end{array}
\right.$
est un endomorphisme et se représente donc dans la base des monômes par la matrice $X$ de taille $d$, appelée usuellement {\bf matrice compagnon}
\begin{equation}
\label{compan}
X =
\begin{bmatrix}
	0 & \cdots & 0 & -a_d/a_0 \\
	1 & 0 & \cdots & -a_{d-1}/a_0 \\
	\vdots  & \ddots  & \ddots & \vdots  \\
	0 & \cdots & 1 & -a_1/a_0
\end{bmatrix}
\end{equation}

Nous avons la proposition classique suivante:
\begin{prop}
\label{compan2roots}
La matrice compagnon admet $f$ comme polynôme caractéristique et comme polynôme minimal, c'est-à-dire que l'on a $f(X) = 0$. De plus les racines du polynôme $f$ sont les valeurs propres de $X$, comptées avec les mêmes multiplicités.
\end{prop}

\begin{rem}
La proposition précédente fournit une méthode efficace de calcul numérique des racines de $f$. En effet, $X$ est une matrice de Hessenberg, à laquelle on peut appliquer de performantes techniques de calcul de valeurs propres, comme la méthode QR. Nous verrons à la section \ref{multivariable} que ces techniques peuvent aussi s'appliquer au cas d'un système multivariable en intersection complète.
\end{rem}

Plus généralement, pour tout élément $g\in A$, l'opérateur de multiplication
$g : \left\vert
\begin{array}{c}
A \mapsto A \\
h \mapsto gh
\end{array}
\right.$
est un endomorphisme, et se représente dans la base des monômes par une matrice appelée encore matrice compagnon de $g$ et qui se calcule facilement à partir de $X$. Considérons par exemple $g = x^2$. L'opérateur de multiplication par $x^2$ n'est autre que le carré de l'opérateur de multiplication par $x$; sa matrice est donc $X^2$. D'une façon générale, nous avons donc
\begin{prop}
la matrice compagnon de $g$ est $g(X)$.
\end{prop}

\begin{rem}
Il faut noter que si $g_1, g_2$ sont deux représentants de $g$ on a $g_1(X) = g_2(X)$, ce qui définit $g(X)$ sans ambiguité, indépendamment du représentant choisi.
\end{rem}

\subsection{Polynômes et matrices de Bezout}
\begin{defn}
\label{def_bez}
Introduisons une nouvelle variable $y$.
Pour tout polynôme $g$, on définit le {\bf polynôme de Bezout $\delta(g)$} et la {\bf matrice de Bezout $B(g) = [b_{\alpha\beta}]$}  par les formules
\begin{equation}
\delta(g) = \dfrac{f(x)g(y)-f(y)g(x)}{x-y} = \sum_{\alpha,\beta = 0, \cdots, m-1} b_{\alpha\beta} x^\alpha y^\beta
\end{equation}
où $m$ désigne n'importe quel entier supérieur ou égal au maximum des degrés de $f$ et $g$.
\end{defn}

\begin{exmp}
Pour $f = x^2 - 3x + 2$, et $g = x^3$ on a les polynômes de Bezout $\delta(1) = -3 + x + y$ et $\delta(x^3) = -2x^2 - 2xy -2y^2 + 3x^2y + 3xy^2 -x^2y^2$ qui, représentés sous forme de tableaux, font apparaitre les matrices de Bezout $B(1)$ et $B(x^3)$
$$
\begin{array}{c|ccc}
\delta(1) & 1 & y & y^2\\
\hline
1 & -3 & 1 & 0\\
x & 1 & 0 & 0\\
x^2 & 0 & 0 & 0
\end{array}
\hspace{1cm}
\begin{array}{c|ccc}
\delta(x^3) & 1 & y & y^2\\
\hline
1 & 0 & 0 & -2\\
x & 0 & -2 & 3\\
x^2 & -2 & 3 & -1
\end{array}
$$
\end{exmp}

\begin{rem}
Le polynôme et la matrice de Bezout sont liés par l'égalité matricielle
\begin{equation}
	\label{pmB}
	\delta(g) = \bold{x} B(g) \bold{y}^T
\end{equation}
où $\bold{x} = (1, x,\cdots, x^{m-1})$ et $\bold{y} = (1, y,\cdots, y^{m-1})$ sont des vecteurs de monômes de $\C[x]$ et $\C[y]$. (Attention, on emploie encore ici la notation ${\bold x}$ pour un vecteur de $\C[x]^m$, notation qui était utilisée précédemment pour désigner la base des monômes ${\bold x}$; en pratique cette confusion n'est pas gênante).
\end{rem}
Considérons maintenant les produits $\bold{x}B(1)$ et $\bold{x}B(g)$. Ces deux familles sont constituées des colonnes de $B(1)$, resp. $B(g)$, vues comme des polynômes en $x$ exprimés dans la base des monômes. On peut voir aussi $\bold{x}B(1)$, resp. $\bold{x}B(g)$, comme la famille des coefficients du polynôme $\delta(1)$, resp. $\delta(g)$, vu comme un polynôme en $y$ à coefficients dans $\C[x]$.

\begin{prop}
\label{relations_prop}
Soit $g$ un polynôme de $\C[x]$, et $m$ le maximum des degrés de $f$ et $g$. Si on écrit $B(1)$ et $B(g)$ dans le même système d'indice $\bold{x} = (1, x,\cdots, x^{m-1})$ et $\bold{y} = (1, y,\cdots, y^{m-1})$, alors dans $A^m$ on a
\begin{equation}
\label{relations}
	\bold{x}B(1)g = \bold{x}B(g)
\end{equation}
\end{prop}
\begin{proof}
Ecrivons
\begin{align} \nonumber
	\delta(g) = g(x)\dfrac{f(x)-f(y)}{x-y} - f(x)\dfrac{g(x)-g(y)}{x-y} \\ \nonumber
	\delta(g) = g(x)\delta(1) - f(x)\dfrac{g(x)-g(y)}{x-y}
\end{align}
Regardons cette dernière égalité comme une égalité entre polynômes en la variable $y$, à coefficients dans $\C[x]$. Si $h\in \C[x][y]$ et $\beta\in\N$ notons $h_\beta$ le coefficient de $y^\beta$ dans $h$. On a alors
$$\delta(g)_\beta = g(x)\delta(1)_\beta - f(x)(\dfrac{g(x)-g(y)}{x-y})_\beta $$
qui est une égalité entre éléments de $\C[x]$. En projetant sur $A$ on a
$\delta(g)_\beta = g(x)\delta(1)_\beta$
et comme ceci est vrai pour tout $\beta\in\N$, on obtient bien la relation~(\ref{relations}).
\end{proof}

\begin{rem}
En disant la  proposition autrement, chaque colonne de $B(1)$ donne, lorsqu'elle est multipliée par $g$ modulo $A$, la colonne de même indice de $B(g)$.
\end{rem}

\begin{exmp}
Reprenant l'exemple précédent, la proposition \ref{relations_prop} dit que, modulo $A$, on a les égalités $(-3 + x)x^3 = -2x^2$, $(1)x^3 = -2x + 3x^2$, $(0)x^3 = -2 + 3x - x^2$, qui se vérifient facilement.
\end{exmp}

\begin{rem}
En considérant les lignes de $B(1), B(x)$ à la place des colonnes on aboutirait à une formule écrite en la variable $y$, identique à la formule~(\ref{relations}) car les matrices de Bezout sont ici symétriques, ce qui ne sera plus le cas en plusieurs variables.
\end{rem}

\subsection{Lien entre matrices de Bezout et matrices compagnon}
\label{Bar}
Particulièrement importantes sont les matrices de Bezout $B(1)$ et $B(x)$
\begin{equation}
	\begin{array}{c|cccc}
		\delta(1) & 1 & y & \dots & y^{d-1} \\
		\hline
		1 & a_{d-1} & \ldots & \dots & a_0 \\
		x & a_{d-2} & \dots & a_0 & 0 \\
		\vdots & \vdots & \vdots & \vdots & \vdots \\
		x_{d-1} & a_0 & 0 & \ldots & 0 \\
	\end{array}
	\hspace{1.5cm}
	\begin{array}{c|cccc}
		\delta(x) & 1 & y & \dots & y^{d-1} \\
		\hline
		1 & -a_{d} & 0 & \dots & 0 \\
		x & 0 & a_{d-2} & \ldots & a_0 \\
		\vdots & \vdots & \vdots & \vdots & \vdots \\
		x_{d-1} & 0 & a_0 & \ldots & 0 \\
	\end{array}
\end{equation}
en effet nous avons le lien suivant entre matrice de Bezout et la matrice compagnon
\begin{prop}
\label{Barnett}
La matrice compagnon $X$ peut se calculer grâce à la {\bf formule de Barnett}
\cite{Barnett}
\begin{equation}
	B(x)B(1)^{-1} = X
\end{equation}
\end{prop}
\begin{proof}
Définissons deux nouveaux familles dans $A$ par
\begin{equation}
	\begin{array}{lll}
		\bold{x}B(1) & = & (a_{d-1} + a_{d-2}x + \cdots + a_0x^{d-1}, \cdots, a_1 + a_0x,  a_0).\\
		\bold{x}B(x) & = & (-a_d, a_{d-2}x + \cdots + a_0x^{d-1}, \cdots, a_0x)
	\end{array}
\end{equation}
et posons $\hat{\bold{x}} = \bold{x}B(1)$.
$B(1)$ étant inversible, la famille $\hat{\bold{x}}$ est une base de $A$ appellée {\bf base de Horner}.
D'après la proposition \ref{relations_prop} on a $\hat{\bold{x}}x = \bold{x}B(1)$. Par construction, les familles $\hat{\bold{x}}$ et $\hat{\bold{x}}x$ s'expriment dans la base $\bold{x}$ (des monômes) respectivement par les matrices $B(1)$ et $B(x)$.
La famille $\hat{\bold{x}}x$ s'exprime donc dans la base $\hat{\bold{x}}$ (de Horner) par la matrice $B(1)^{-1}B(x)$ ce qui veut dire que l'endomorphisme
$x : \left\vert
\begin{array}{c}
A \mapsto A \\
h \mapsto xh
\end{array}
\right.$ a pour matrice $B(1)^{-1}B(x)$ dans la base $\hat{\bold{x}}$
et pour matrice $B(1)(B(1)^{-1}B(x))B(1)^{-1} = B(x)B(1)^{-1}$ dans la base $\bold{x}$.
\end{proof}

\subsection{Formule de Barnett généralisée}
\label{Bar_gen}
La formule de Barnett a été écrite en considérant les matrices de Bezout des polynômes $1$ et $x$.
Si on considère un polynôme quelconque $g$ de $\C[x]$ et $B(g)$ sa matrice de Bezout il serait naturel d'avoir entre les matrices de Bezout $B(1)$ et $B(g)$ la relation suivante, que nous appellerons {\bf formule de Barnett généralisée}
\begin{equation}
	\label{BG}
	B(g)B(1)^{-1} = g(X)
\end{equation}
On montre facilement la formule~(\ref{BG}) lorsque le degré de g est inférieur ou égal à $d$, c'est-à-dire lorsque $B(1)$ et $B(g)$ sont de même taille. Par exemple pour $f = x^2 - 3x + 2$, $d = 2$, on a
$$
\begin{array}{c|cc}
	\delta(1) & 1 & y \\
	\hline
	1 & -3 & 1 \\
	x & 1 & 0
\end{array}
\hspace{1cm}
\begin{array}{c|cc}
	\delta(x) & 1 & y \\
	\hline
	1 & -2 & 1 \\
	x & 1 & 0
\end{array}
\hspace{1cm}
\begin{array}{c|cc}
	\delta(x^2) & 1 & y \\
	\hline
	1 & 0 & -2 \\
	x & -2 & 3
\end{array}
$$
\begin{equation}
	B(x)B(1)^{-1} =
	\begin{bmatrix}
		0 & -2 \\
		1 & 3
	\end{bmatrix}
	= X
	\hspace{1cm}
	B(x^2)B(1)^{-1} =
	\begin{bmatrix}
		-3 & -6 \\
		2 & 7
	\end{bmatrix}
	= X^2
\end{equation}
ce qui confirme bien la formule~(\ref{BG}).
Si par contre le degré de $g$ est supérieur à $d$, alors $B(g)$ et $B(1)$ ne sont plus de la même taille et une opération telle que $B(g)B(1)^{-1}$ n'a plus de sens. Une première idée est de réécrire les deux matrices de Bezout dans le même système d'indices, à savoir
$\bold{x} = (1, x,\cdots, x^{m-1})$ et $\bold{y} = (1, y,\cdots, y^{m-1})$, $m$ étant le degré de $g$. Par exemple en choisissant $f$ comme ci-dessus et $g = x^3$ on aurait
$$
\begin{array}{c|ccc}
\delta(1) & 1 & y & y^2\\
\hline
1 & -3 & 1 & 0\\
x & 1 & 0 & 0\\
x^2 & 0 & 0 & 0
\end{array}
\hspace{1cm}
\begin{array}{c|ccc}
\delta(x^3) & 1 & y & y^2\\
\hline
1 & 0 & 0 & -2\\
x & 0 & -2 & 3\\
x^2 & -2 & 3 & -1
\end{array}
$$
mais alors $B(1)$ n'est plus inversible. Nous allons cependant montrer, grâce aux relations~(\ref{relations}), que si on projette les deux polynômes de Bezout sur le quotient $A$ alors les matrices $B(g)$ et $B(1)$ sont redimensionnées à la même taille, $B(1)$ est inversible et la formule~(\ref{BG}) s'applique. Illustrons le procédé sur l'exemple ci-dessus. Puisque $B(1)$ n'est pas inversible, on peut trouver une combinaison linéaire de colonnes qui s'annule, ici c'est la troisième colonne qui est nulle. En la multipliant par $x^3$, et en appliquant les relations~(\ref{relations}), on obtient que la troisième colonne de $B(x^3)$ est aussi nulle. Mais cette colonne vaut $-2 + 3x - x^2$ ce qui entraine que, dans le quotient,
$-2 + 3x - x^2 = 0$ (ce n'est pas une surprise car ce dernier polynôme n'est autre que $-f$; ceci est dû au fait que l'exemple choisi est particulièrement simple, mais nous verrons dans le cas multivariable que les relations nulles dans le quotient ainsi générées sont loin d'être triviales). En vue d'automatiser les calculs, traduisons le procédé précédent en termes d'algèbre matricielle.
Toujours sur le même exemple
\begin{align} \nonumber 
	\delta(x^3) &=
	\begin{bmatrix}
			1 & x & x^2
	\end{bmatrix}
	\begin{bmatrix}
		0 & 0 & -2 \\
		0 & -2 & 3 \\
		-2 & 3 & -1
	\end{bmatrix}
	\begin{bmatrix}
		1 \\
		y \\
		y^2
	\end{bmatrix} \\ \nonumber 
	\delta(x^3) &=
	\begin{bmatrix}
		1 & x & x^2
	\end{bmatrix}
	\begin{bmatrix}
		1 & 0 & 2 \\
		0 & 1 & -3 \\
		0 & 0 & 1
	\end{bmatrix}
	\begin{bmatrix}
		1 & 0 & -2 \\
		0 & 1 & 3 \\
		0 & 0 & 1
	\end{bmatrix}
	\begin{bmatrix}
		0 & 0 & -2 \\
		0 & -2 & 3 \\
		-2 & 3 & -1
	\end{bmatrix}
	\begin{bmatrix}
		1 \\
		y \\
		y^2
	\end{bmatrix} \\ \nonumber 
	\delta(x^3) &=
	\begin{bmatrix}
			1 & x & 2 - 3x + x^2
	\end{bmatrix}
	\begin{bmatrix}
		4 & -6 & 0 \\
		-6 & 7 & 0 \\
		-2 & 3 & -1
	\end{bmatrix}
	\begin{bmatrix}
		1 \\
		y \\
		y^2
	\end{bmatrix} \\ \nonumber 
\end{align}
En résumé, nous multiplions le vecteur d'indices
$\begin{bmatrix}
	1 & x & x^2
\end{bmatrix}$ à droite par la transformation de Gauss
$P =
\begin{bmatrix}
	1 & 0 & 2 \\
	0 & 1 & -3 \\
	0 & 0 & 1
\end{bmatrix}$
et les deux matrices de Bezout $B(1)$ et $B(g)$ à gauche par $P^{-1}$. Les polynômes de Bezout, écrits sous forme de tableaux, deviennent alors
$$
\begin{array}{c|ccc}
	\delta(1) & 1 & y & y^2\\
	\hline
	1 & -3 & 1 & 0\\
	x & 1 & 0 & 0\\
	2 - 3x + x^2 & 0 & 0 & 0
\end{array}
\hspace{1cm}
\begin{array}{c|ccc}
	\delta(x^3) & 1 & y & y^2\\
	\hline
	1 & 4 & -6 & 0 \\
	x & -6 & 7 & 0 \\
	2 - 3x + x^2 & -2 & 3 & -1
\end{array}
$$
Ce que disent les relations~(\ref{relations}) c'est que la troisième colonne de $B(x^3)$ est nulle dans le quotient $A$, c'est à-dire $-2 + 3x - x^2 = 0$, (on reconnait l'égalité $-f = 0$). On a donc
$$
\delta(1) = \begin{bmatrix}
	1 & x
\end{bmatrix}
\begin{bmatrix}
	-3 & 1 \\
	1 & 0
\end{bmatrix}
\begin{bmatrix}
	1 \\
	y
\end{bmatrix}$$
$$\delta(x^3) = \begin{bmatrix}
	1 & x
\end{bmatrix}
\begin{bmatrix}
	4 & -6 \\
	-6 & 7
\end{bmatrix}
\begin{bmatrix}
	1 \\
	y
\end{bmatrix} + (2 - 3x + x^2)(-2 + 3y - y^2)$$
puis, en projetant $\delta(1), \delta(g)$ sur $A \otimes A$
$$
\begin{array}{c|cc}
	\delta(1) & 1 & y \\
	\hline
	1 & -3 & 1 \\
	x & 1 & 0
\end{array}
\hspace{1cm}
\begin{array}{c|cc}
	\delta(x^3) & 1 & y \\
	\hline
	1 & 4 & -6  \\
	x & -6 & 7
\end{array}
$$
Nous avons bien obtenu des matrices de Bezout de même taille, avec $B(1)$ inversible. Formons alors le quotient
\begin{equation}
	B(x^3)B(1)^{-1} =
	\begin{bmatrix}
		-6 & -14 \\
		7 & 15
	\end{bmatrix}
	= X^3
\end{equation}
ce qui est bien conforme à la formule de Barnett généralisée.

\begin{rem}
On peut remplacer la matrice de Gauss par toute matrice permettant de transformer une colonne donnée en une colonne possédant un seul élément non nul, comme par exemple une matrice orthogonale de Householder. C'est le choix qui sera fait dans l'implémentation en Octave proposée en \cite{jp_code}.
\end{rem}

\section{Cas multivariable}
\label{multivariable}

Dans le cas univariable, examiné à la section précédente, la structure de $A$ se compose d'une part d'une base, en l'occurence la base des monômes, d'autre part de la matrice compagnon, exprimant l'endomorphisme de multiplication par $x$ dans cette base. Ces deux éléments peuvent être obtenus soit directement par lecture des coefficients de $f$, soit à partir des matrices de Bezout $B(1), B(x)$.

Dans le le cas multivariable, développé dans cette section, ni une base ni les matrices compagnon (matrices des opérateurs
$x_j : \left\vert
\begin{array}{c}
A \mapsto A \\
h \mapsto x_jh
\end{array}
\right.$ dans la base) ne sont visibles directement sur les coefficients des polynômes de départ. En revanche nous allons montrer comment construire des matrices de Bezout $B(1), B(x_1), \cdots, B(x_n)$ à partir desquelles on peut obtenir une base et les matrices compagnon $X_j$ associées à la base obtenue. Commençons par fixer le cadre de travail.
Pour $n$ polynômes $f_1,\cdots, f_n$ en $n$ variables $x_1,\cdots, x_n$ à coefficients dans $\C$, considérons :
\begin{itemize}
\item $\C[x]$ l'anneau des polynômes en les variables $x = x_1,\cdots, x_n$
\item $\langle f \rangle$ l'idéal généré par la famille  $f = f_1,\cdots, f_n$
\item $V(f) = \{x \in \C^n : f(x) = 0\}$ la variété associée à $\langle f\rangle$
\item $A = \C[x]/\langle f\rangle$ l'algèbre quotient
\end{itemize}
Nous supposerons dorénavant que l'idéal $\langle f\rangle$ est {\bf zéro-dimensionel}, c'est-à-dire que $V(I)$ est fini ou, de façon équivalente \cite[p.~234]{clo}, que $A$ est de {\bf dimension finie} en tant qu'espace vectoriel sur $\C$. Ceci est bien sûr toujours le cas lorsque $n = 1$.

\subsection{Construction des polynômes et des matrices de Bezout}

\subsubsection{Extension de la définition \ref{def_bez} au cas multivariable}

\begin{defn}
Soit $x^\gamma = x_1^{\gamma_1}\cdots x_n^{\gamma_n} \in \C[x]$ un monôme.
Introduisons un nouveau jeu de variables $y = y_1,\cdots, y_n$ et considérons, pour tous $i, j = 1\cdots n$, le rapport
\begin{equation}
\label{finite_diff}
\delta_{i,j}(x^\gamma) = \dfrac{y_j^{\gamma_j}f_i(y_1,\cdots, y_{j-1},x_j,\cdots,x_n) - x_j^{\gamma_j}f_i(y_1,\cdots,y_j,x_{j+1},\cdots,x_n)}{x_j - y_j}
\end{equation}
qui est un polynôme en les variables $x, y$. Nous obtenons une matrice de différences finies $\Delta(x^\gamma) = (\delta_{ij}(x^\gamma))_{ij}$, qui est à la matrice jacobienne ce que le taux d'accroissement est à la dérivée.
Le {\bf polynôme de Bezout} du monôme $x^\gamma$ est par définition
\begin{equation}
	\delta(x^\gamma) = det(\Delta(x^\gamma))
\end{equation}
qui est un élément de $\C[x, y]$. Pour un polynôme général $g = \sum_\gamma g_\gamma x^\gamma \in \C[x]$, on étend la définition précédente par linéarité $\delta(g) = \sum_\gamma g_\gamma \delta(x^\gamma)$.
En développant $\delta(g) = \sum_{\alpha,\beta} b_{\alpha\beta} x^\alpha y^\beta$ comme une somme de monômes de $\C[x, y]$, et en notant $\bold{x}$ et $\bold{y}$ les familles de tous les monômes en $x$ et $y$ apparaissant dans ce développement, nous définissons la {\bf matrice de Bezout} $B(g) = [b_{\alpha\beta}]$, c'est à dire que l'on a la relation suivante, similaire à la relation~(\ref{pmB}), entre polynôme et matrice de Bezout
\begin{equation}
	\delta(g) = \bold{x} B(g) \bold{y}^T
\end{equation}
\end{defn}
Illustrons la définition précédente à l'aide d'un exemple, tiré de \cite{jpc}.
\begin{exmp}
\label{bez_multi}
Fixons $n = 2$ et considérons $f_1 = x_1^2 + x_1x_2^2 - 1, f_2 = x_1^2x_2 + x_1$.
Nous allons calculer les matrices de Bezout $B(1), B(x_1), B(x_2)$,  qui vont servir à la construction des matrices compagnon $X_1, X_2$, comme nous le verrons plus loin. Pour commencer, calculons à partir des formules (\ref{finite_diff})
\begin{align}
\Delta(1) &=
\begin{pmatrix}
x_1 + x_2^2 + y_1 & x_2y_1 + y_1y_2 \\
1 + x_1x_2 + x_2y_1 & y_1^2
\end{pmatrix} \nonumber  \\
\Delta(x_1) &=
\begin{pmatrix}
1 + x_1y_1 & x_2y_1 + y_1y_2 \\
1 + x_1x_2 + x_2y_1 & y_1^2
\end{pmatrix} \nonumber  \\
\Delta(x_2) &=
\begin{pmatrix}
x_1 + x_2^2 + y_1 & 1 - y_1^2 + x_2y_1y_2 \\
1 + x_1x_2 + x_2y_1  & -y_1
\end{pmatrix} \nonumber
\end{align}
dont le déterminant fournit les polynômes de Bezout
\begin{align}
\delta(1) &= -x_2y_1 - x_1x_2^2y_1 + x_1y_1^2 + y_1^3 - y_1y_2 - x_1x_2y_1y_2 - x_2y_1^2y_2 \nonumber \\
\delta(x_1) &=  y_1^2 - x_1x_2^2y_1^2 + x_1y_1^3 - x_1x_2y_1^2y_2 \nonumber \\
\delta(x_2) &= -1 - x_1x_2 - x_1y_1 -x_2y_1 - x_2^2y_1 + x_1x_2y_1^2 + x_2y_1^3 - x_2y_1y_2 - x_1x_2^2y_1y_2 - x_2^2y_1^2y_2\nonumber
\end{align}
Les familles de mônomes apparaissant dans ces polynômes sont
$\bold{x} = (1, x_2, x_2^2, x_1, x_1x_2, x_1x_2^2)$ et $\bold{y} = (1, y_1, y_1y_2, y_1^2, y_1^2y_2, y_1^3)$.
Les polynômes de Bezout s'écrivent sous forme de tableaux faisant apparaitre les matrices de Bezout
$$\begin{array}{c|cccccc}
	\delta(1) & 1 & y_1 & y_1y_2 & y_1^2 & y_1^2y_2 & y_1^3 \\
	\hline
	1 &  &  & -1 &  &  & 1\\
	x_2 &  & -1 &  &  & -1 & \\
	x_2^2 &  &  &  &  &  & \\
	x_1 &  &  &  & 1 &  & \\
	x_1x_2 &  &  & -1 &  &  & \\
	x_1x_2^2 &  & -1 &  &  &  &
\end{array}$$

$$\begin{array}{c|cccccc}
	\delta(x_1) & 1 & y_1 & y_1y_2 & y_1^2 & y_1^2y_2 & y_1^3 \\
	\hline
	1 &  &  &  & 1 &  & \\
	x_2 &  &  &  &  &  & \\
	x_2^2 &  &  &  &  &  & \\
	x_1 &  &  &  &  &  & 1\\
	x_1x_2 &  &  &  &  & -1 & \\
	x_1x_2^2 &  &  &  & -1 &  &
\end{array}
\hspace{0.2cm}
\begin{array}{c|cccccc}
	\delta(x_2) & 1 & y_1 & y_1y_2 & y_1^2 & y_1^2y_2 & y_1^3 \\
	\hline
	1 & -1 &  &  &  &  & \\
	x_2 &  & -1 & -1 &  &  & 1\\
	x_2^2 &  & -1 &  &  & -1 & \\
	x_1 &  & -1 &  &  &  & \\
	x_1x_2 & -1 &  &  & 1 &  & \\
	x_1x_2^2 &  &  & -1 &  &  &
\end{array}$$

\end{exmp}
\begin{rem}
Ici, contrairement au cas univariable, les listes $\bold{x}$ et $\bold{y}$ ne sont pas des bases de $A$. Nous verrons plus loin qu'elles sont cependant génératrices et comment on peut en extraire des bases.
\end{rem}

\subsubsection{Calcul effectif des matrices de Bezout}
Dans l'exemple précédent, les polynômes de Bezout s'obtiennent en calculant le déterminant des matrices $\Delta(1), \Delta(x_1), \Delta(x_2)$, qui sont de taille $2$ et dont les coefficients sont des polynômes en $x_1, x_2$. Si le nombre de variables $n$ ou le degré des polynômes $f_i$ augmente alors ce calcul peut devenir difficile car les coefficients des matrices $\Delta(x_k)$ ne sont pas numériques et on ne peut donc pas appliquer la méthode du pivot de Gauss. Un moyen de résoudre cette difficulté est de procéder par évaluation-interpolation :
\begin{enumerate}
\item
on estime à priori l'ensemble des monômes qui vont apparaitre dans le résultat $\delta(x^\gamma)$
\item
on évalue $\Delta(x_k)$ sur un ensemble adéquat $U \times V$ de multi-points de Fourier $u = (u_1,\cdots, u_n) \in U$ et $v = (v_1,\cdots, v_n) \in V$
\item
pour chaque point $(u, v) \in U\times V$ on calcule le déterminant numérique de $\Delta(x_k)(u, v)$ par la méthode du pivot de Gauss.
\item
pour finir on interpole l'ensemble des valeurs obtenues par le polynôme cherché $\delta(x_k)$.
\end{enumerate}
Pour implémenter cet algorithme concrètement on doit préciser l'ensemble des monômes de $\delta(x_k)$ ainsi que les points de Fourier utilisés pour l'évaluation de $\delta(x_k)$. Prenons l'exemple d'un système polynomial $f$ de multidegré $(d_1, \cdots, d_n)$, c'est-à-dire que pour tous $i, j = 1..n$ le degré de $f_i$ en la variable $x_j$ est inférieur ou égal à $d_j$. Fixons un entier $k$ compris entre $0$ et $n$ et adoptons la convention que $x_0 = 1$. On voit facilement que $\delta(x_k)$, polynôme en $x, y$, est de multidegré $(d_1, 2d_2, \cdots, nd_n)$ en $x$ et de multidegré $(nd_1, (n-1)d_2, \cdots, d_n)$ en $y$.
 Pour l'évaluation de $\delta(x_k)$ aux points de Fourier $(u, v) \in U\times V$ nous choisirons donc $U = \prod_{j=1..n} U_j$ o\`u $U_j$ est l'ensemble des racines complexes de $X^{jd_j} - 1$. De même nous choisirons $V = \prod_{j=1..n} V_j$ de façon que $U_j$ et $V_j$ soient disjoints, afin que le dénominateur ne s'annule jamais dans la formule (\ref{finite_diff}). Ceci est réalisé par exemple lorsque $V_j$ est l'ensemble des racines complexes de $X^{(n-j+1)d_j} - \theta_j$ avec $\theta_j = e^{i\pi/j}$. Les considérations précédentes nous permettent maintenant d'écrire l'algorithme~\ref{fourierPoints} présenté ci-dessous fournissant les ensembles $U$ et $V$.
\begin{algorithm}
\caption{Construction des ensembles $U, V$ de points de Fourier servant à l'évaluation du polynôme de Bezout $\delta(x_k)$, $k = 0, \cdots, n$.}
\label{fourierPoints}
\begin{algorithmic}
\Function{fourierPoints}{$d$} \Comment{multidegré $d = (d_1,\cdots,d_n)$}
\For{$j=1..n$} \Comment{construction des facteurs $U_j, V_j$}
\State $U_j \gets$ ensemble des racines de $X^{jd_j}-1$
\State $V_j \gets$ ensemble des racines de $X^{(n-j+1)d_j}-e^{i\pi/j}$
\EndFor
\State $U \gets \prod_{j=1..n}U_j$
\State $V \gets \prod_{j=1..n}V_j$
\State \textbf{return} $U, V$
\EndFunction
\end{algorithmic}
\end{algorithm}
Les ensembles de points de Fourier $U$ et $V$ peuvent alors être utilisés par l'algorithme \ref{evaluationMatrix} suivant pour construire la matrice d'évaluation du polynôme de Bezout.
\begin{algorithm}[H]
\caption{Construction de la matrice $C^{(k)}$ d'évaluation du polynôme de Bezout $\delta(x_k)$}
\label{evaluationMatrix}
\begin{algorithmic}
\Function{evaluation}{$f, k$} \Comment{$f = (f_1,\cdots,f_n)$ système polynomial}
\State $U, V \gets \textsc{fourierPoints}(d)$ \Comment{$d$ multidegré de $f$}
\State $D \gets \prod_{j=1..n}jd_j$
\State $C^{(k)} \gets \textsc{zeros}(D, D)$
\For{$(u, v) \in U\times V$}
      \State $\Delta \gets \textsc{zeros}(n, n)$
   		\For{$i, j=1..n$}
      		\State $\Delta_{i,j} \gets \delta_{i,j}(x_k)(u, v)$ \Comment{$\delta_{i,j}(x_k)$ défini à la formule (\ref{finite_diff})}
   		\EndFor
		\State $C^{(k)}_{u, v} \gets \textsc{det}(\Delta)$
	\EndFor
\State \textbf{return} $C^{(k)}$
\EndFunction
\end{algorithmic}
\end{algorithm}
Ayant noté $C^{(k)}$ la matrice d'évaluation du polynôme de Bezout $\delta(x_k)$, notons $B^{(k)}$ la matrice de Bezout $B(x_k)$.
La matrice $B^{(k)} = \left[b^{(k)}_{\alpha\beta}\right]_{\alpha\beta}$ est définie par $\delta^{(k)}(x, y) = \sum_{\alpha,\beta} b^{(k)}_{\alpha\beta} x^\alpha y^\beta$. On a donc $C^{(k)}_{u,v} = \delta^{(k)}(u, v) = \sum_{\alpha,\beta} b^{(k)}_{\alpha\beta} u^\alpha v^\beta$, ce qui s'écrit comme produit de matrices
$\left[C^{(k)}_{u,v}\right]_{u,v} = \left[u^\alpha\right]_{u,\alpha} \left[b^{(k)}_{\alpha,\beta}\right]_{\alpha, \beta} \left[v^\beta\right]_{v, \beta}^T$. Définissons alors les matrices de Fourier $F_u = \left[ u^\alpha \right]_{u, \alpha}$
 et $F_v = \left[ v^\beta \right]_{v, \beta}$. On obtient la relation d'évaluation-interpolation entre les matrices $B^{(k)}$ et $C^{(k)}$
$$C^{(k)} = F_uB^{(k)} F_v^T$$
 Grâce au choix des points de Fourier fait dans l'algorithme \ref{fourierPoints} les matrices $F_u$ et $F_v$ sont unitaires et la matrice $B^{(k)}$ s'obtient alors facilement par la relation
 \begin{equation}
 B^{(k)} = F_u^*C^{(k)} \overline{F_v}
 \end{equation}
La construction des matrices de Bezout, inspirée des considérations précédentes, a été implémentée en Numpy et publiée sur le site \cite{jp_code}.

\subsection{Formules de Barnett et structure de l'algèbre quotient.}
Dans l'hypothèse où l'idéal est zéro-dimensionel, la dimension de l'algèbre quotient $A = \C[\bold{x}]/\langle f\rangle$ est finie, ce qui assure l'existence d'une base et de matrices compagnon $X_1,\cdots, X_n$ (matrices des opérateurs de multiplication par les variables dans la base considérée). Nous allons montrer que le même procédé mis en oeuvre dans la section \ref{Bar_gen}, consistant en manipulations sur les colonnes des matrices de Bezout $B(1), B(x_1), \cdots, B(x_n)$, permet ici aussi de construire une base de $A$ ainsi que les matrices compagnons associées. Rappelons tout d'abord un certain nombre de propriétés algébriques du polynôme $\delta(1)$ et des matrices de Bezout $B(x_k)$.

\subsubsection{Propriétés algébriques du polynôme $\delta(1)$ et de la matrice $B(1)$}
Les propriétés qui suivent sont de nature algébrique et sont données sans démonstration. Le lecteur intéressé pourra consulter les détails dans \cite{jpc}. Comme dans la proposition \ref{Barnett}, définissons de nouvelles familles d'éléments de $A$ par les produits vecteur-matrice :
\begin{equation}
		\hat{\bold{x}}_k  =  \bold{x}B(x_k), \quad k=0\cdots n
\end{equation}
avec la convention de notation habituelle $\hat{\bold{x}}_0 = \hat{\bold{x}}$.
\begin{exmp}
En reprenant l'exemple \ref{bez_multi} nous avons
\begin{equation}
	\begin{array}{lll}
		\hat{\bold{x}}_0 & = & (0, -x_2 - x_1x_2^2, -1 - x_1x_2, x_1, -x_2, 1) \\
		\hat{\bold{x}}_1 & = & (0, 0, 0, -1 - x_2^2, -x_1x_2, x_1) \\
		\hat{\bold{x}}_2 & = & (-1 - x_1x_2, -x_2 - x_2^2 - x_1, - x_2 - x_1x_2^2, x_1x_2, -x_2^2, x_2)
	\end{array}
\end{equation}
\end{exmp}

\begin{prop}
\label{xj} (admise, démonstration dans \cite{jpc}).
Pour tout $k=1\cdots n$ on a
\begin{equation}
    \hat{\bold{x}}_0x_k = \hat{\bold{x}}_k
\end{equation}
\end{prop}
Les relations ci-dessus sont faciles à vérifier sur l'exemple \ref{bez_multi}. Jusqu'à maintenant les cas univariable et multivariable sont très similaires, sauf sur un point: dans le cas multivariable les familles $\bold{x}$ et $\hat{\bold{x}}$ ne sont plus nécessairement des bases de $A$. On a cependant la propriété suivante

\begin{prop} (admise, démonstration dans \cite{jpc}).
Chacune des familles $\bold{x}$ et $\hat{\bold{x}}$ est génératrice dans $A$.
\end{prop}

\subsubsection{Processus de réduction}
\label{sec:reduction_process}
La proposition précédente fournit un début de structure de l'algèbre $A$. Nous avons pour l'instant une famille génératrice $\bold{x}$ de $A$ ainsi que des matrices de Bezout $B(x_k), k = 0, \cdots, n$. Nous allons montrer comment, en appliquant le procédé matriciel décrit dans la section \ref{Bar_gen} à la famille génératrice $\bold{x}$ et aux matrices de Bezout $B(x_k)$, on peut fabriquer une base de $A$ et des matrices compagnon $X_k$. Illustrons les calculs à partir de l'exemple \ref{bez_multi}. Le rang de $B(1)$ est $5$. La première colonne de $B(x_1)$ est nulle mais celle de $B(x_2)$ ne l'est pas, ce qui fournit la relation dans le quotient $1 + x_1x_2 = 0$.
Multiplions $\bold{x}$ à droite par la matrice de Gauss $P$ dont la cinquième colonne vaut $(1, 0, 0, 0, 1, 0)^{T}$, et multiplions les matrices de Bezout à gauche par $P^{-1}$, ce qui revient à soustraire la cinquième ligne à la première.
Les matrices de Bezout $B(1), B(x_1), B(x_2)$ s'écrivent:
$$
\begin{array}{c|cccccc}
	B(1) & 1 & y_1 & y_1y_2 & y_1^2 & y_1^2y_2 & y_1^3 \\
	\hline
	1 &  &  &  &  &  & 1\\
	x_2 &  & -1 &  &  & -1 & \\
	x_2^2 &  &  &  &  &  & \\
	x_1 &  &  &  & 1 &  & \\
	1+x_1x_2 &  &  & -1 &  &  & \\
	x_1x_2^2 &  & -1 &  &  &  &
\end{array}$$
$$
\begin{array}{c|cccccc}
	B(x_1) & 1 & y_1 & y_1y_2 & y_1^2 & y_1^2y_2 & y_1^3 \\
	\hline
	1 &  &  &  & 1 & 1 & \\
	x_2 &  &  &  &  &  & \\
	x_2^2 &  &  &  &  &  & \\
	x_1 &  &  &  &  &  & 1\\
	1+x_1x_2 &  &  &  &  & -1 & \\
	x_1x_2^2 &  &  &  & -1 &  &
\end{array}
\hspace{0.2cm}
\begin{array}{c|cccccc}
	B(x_2) & 1 & y_1 & y_1y_2 & y_1^2 & y_1^2y_2 & y_1^3 \\
	\hline
	1 &  &  &  & -1 &  & \\
	x_2 &  & -1 & -1 &  &  & 1\\
	x_2^2 &  & -1 &  &  & -1 & \\
	x_1 &  & -1 &  &  &  & \\
	1+x_1x_2 & -1 &  &  & 1 &  & \\
	x_1x_2^2 &  &  & -1 &  &  &
\end{array}
$$
La première colonne de $B(x_2)$ contient maintenant un seul coefficient non nul, indexé par $1 + x_1x_2$. On peut donc, en projetant les trois bezoutiens sur $A_x$, supprimer la première colonne et la cinquième ligne dans les trois matrices:
$$
\begin{array}{c|ccccc}
	B(1) & y_1 & y_1y_2 & y_1^2 & y_1^2y_2 & y_1^3 \\
	\hline
	1  &  &  &  &  & 1 \\
	x_2  & -1 &  &  & -1 & \\
	x_2^2  &  &  &  &  & \\
	x_1  &  &  & 1 &  & \\
	x_1x_2^2  & -1 &  &  &  &
\end{array}$$
$$
\begin{array}{c|ccccc}
	B(x_1)  & y_1 & y_1y_2 & y_1^2 & y_1^2y_2 & y_1^3 \\
	\hline
	1  &  &  & 1 & 1 & \\
	x_2  &  &  &  &  & \\
	x_2^2  &  &  &  &  & \\
	x_1  &  &  &  &  & 1 \\
	x_1x_2^2  &  &  & -1 &  &
\end{array}
\hspace{0.2cm}
\begin{array}{c|ccccc}
	B(x_2) & y_1 & y_1y_2 & y_1^2 & y_1^2y_2 & y_1^3 \\
	\hline
	1  &  &  & -1 &  & \\
	x_2  & -1 & -1 &  &  & 1 \\
	x_2^2  & -1 &  &  & -1 & \\
	x_1  & -1 &  &  &  & \\
	x_1x_2^2 &  & -1 &  &  &
\end{array}
$$

La deuxième colonne de $B(1)$ est nulle, celle de $B(x_2)$ ne l'est pas. La relation est $x_2 + x_1x_2^{2} = 0$. La matrice $P$ est définie par sa cinquième colonne $(0, 1, 0, 0, 1)^{T}$. Le vecteur $\bold{x}$ devient $(1, x_2, x_2^{2}, x_1, x_2 + x_1x_2^{2})$. On soustrait la cinquième ligne à la deuxième. Les bezoutiens s'écrivent:
$$
\begin{array}{c|ccccc}
	B(1) & y_1 & y_1y_2 & y_1^2 & y_1^2y_2 & y_1^3 \\
	\hline
	1  &  &  &  &  & 1 \\
	x_2  &  &  &  & -1 & \\
	x_2^2  &  &  &  &  & \\
	x_1  &  &  & 1 &  & \\
	x_2 + x_1x_2^2  & -1 &  &  &  &
\end{array}$$
$$
\begin{array}{c|ccccc}
	B(x_1)  & y_1 & y_1y_2 & y_1^2 & y_1^2y_2 & y_1^3 \\
	\hline
	1  &  &  & 1 & 1 & \\
	x_2  &  &  & 1 &  & \\
	x_2^2  &  &  &  &  & \\
	x_1  &  &  &  &  & 1 \\
	x_2 + x_1x_2^2  &  &  & -1 &  &
\end{array}
\hspace{0.2cm}
\begin{array}{c|ccccc}
	B(x_2) & y_1 & y_1y_2 & y_1^2 & y_1^2y_2 & y_1^3 \\
	\hline
	1  &  &  & -1 &  & \\
	x_2  & -1 &  &  &  & 1 \\
	x_2^2  & -1 &  &  & -1 & \\
	x_1  & -1 &  &  &  & \\
	x_2 + x_1x_2^2 &  & -1 &  &  &
\end{array}
$$

La deuxième colonne de $B_2$ contient un seul coefficient non nul, en cinquième ligne, on peut donc supprimer les deuxièmes colonnes et les cinquièmes lignes:

$$
\begin{array}{c|cccc}
	B(1) & y_1 & y_1^2 & y_1^2y_2 & y_1^3 \\
	\hline
	1  &   &  &  & 1 \\
	x_2  &  &  & -1 & \\
	x_2^2  &  &  &  & \\
	x_1  &  & 1 &  &
\end{array}
\hspace{0.2cm}
\begin{array}{c|cccc}
	B(x_1)  & y_1 & y_1^2 & y_1^2y_2 & y_1^3 \\
	\hline
	1  &  & 1 & 1 & \\
	x_2  &  & 1 &  & \\
	x_2^2  &  &  &  & \\
	x_1  &  &  &  & 1
\end{array}
\hspace{0.2cm}
\begin{array}{c|cccc}
	B(x_2) & y_1 & y_1^2 & y_1^2y_2 & y_1^3 \\
	\hline
	1  &  & -1 &  & \\
	x_2  & -1 &  &  & 1 \\
	x_2^2  & -1 &  & -1 & \\
	x_1  & -1 &  &  &
\end{array}
$$
La première colonne de $B(1)$ est nulle, celle de $B(x_2)$ ne l'est pas. La relation est $x_2 + x_2^{2} + x_1 = 0$. La matrice $P$ est définie par sa quatrième colonne $(0, 1, 1, 1)^{T}$. Le vecteur $\bold{x}$ devient $(1, x_2, x_2^{2},  x_2 + x_2^{2} + x_1)$. On soustrait la quatrième ligne à la deuxième et à la troisième. Les bezoutiens s'écrivent:

$$
\begin{array}{c|cccc}
	B(1) & y_1 & y_1^2 & y_1^2y_2 & y_1^3 \\
	\hline
	1  &   &  &  & 1 \\
	x_2  &  & -1 & -1 & \\
	x_2^2  &  & -1 &  & \\
	x_2 + x_2^{2} + x_1  &  & 1 &  &
\end{array}$$
$$
\begin{array}{c|cccc}
	B(x_1)  & y_1 & y_1^2 & y_1^2y_2 & y_1^3 \\
	\hline
	1  &  & 1 & 1 & \\
	x_2  &  &  &  & \\
	x_2^2  &  &  &  & -1 \\
	x_2 + x_2^{2} + x_1  &  &  &  & 1
\end{array}
\hspace{0.2cm}
\begin{array}{c|cccc}
	B(x_2) & y_1 & y_1^2 & y_1^2y_2 & y_1^3 \\
	\hline
	1  &  & -1 &  & \\
	x_2  &  &  &  & 1 \\
	x_2^2  &  &  & -1 & \\
	x_2 + x_2^{2} + x_1  & -1 &  &  &
\end{array}$$

La première colonne de $B_2$ contient un seul coefficient non nul, en quatrième ligne, on peut donc supprimer les premières colonnes et les quatrièmes lignes:
$$
\begin{array}{c|ccc}
	B(1) & y_1^2 & y_1^2y_2 & y_1^3 \\
	\hline
	1  &  &  & 1 \\
	x_2  & -1 & -1 & \\
	x_2^2 & -1 &  &
\end{array}
\hspace{0.2cm}
\begin{array}{c|ccc}
	B(x_1) & y_1^2 & y_1^2y_2 & y_1^3 \\
	\hline
	1  & 1 & 1 & \\
	x_2  & 1 &  & -1\\
	x_2^2  &  &  & -1
\end{array}
\hspace{0.2cm}
\begin{array}{c|ccc}
	B(x_2) & y_1^2 & y_1^2y_2 & y_1^3 \\
	\hline
	1  & -1 &  & \\
	x_2  &  &  & 1 \\
	x_2^2  &  & -1 &
\end{array}$$
A ce stade, la matrice $B(1)$ est inversible et le processus de réduction est donc terminé. On vérifie que les familles $\bold{x} = (1, x_2, x_2^{2})$ et $\bold{y} = (y_1, y_1^{2}, y_1^{3})$ sont des bases de $A$, dont les bases de Horner associées sont $\hat{\bold{x}} = (-x_2-x_2^{2}, -x_{2}, 1)$ et
$\hat{\bold{y}} = (y_1^{3}, -y_1^{2}-y_1^{2}y_2, -y_1^{2})$. La dimension de $A$ est ici égale à~$3$.
D'une façon générale nous avons (\cite{jpc} p.57, \cite{bm}, \cite{tm})
\begin{prop}
	\label{conjecture}
Lorsque le processus de réduction est terminé, c'est-à-dire lorsque la matrice $B(1)$ est inversible et que toutes les matrices $B(x_k), k=0, \cdots, n$ sont de même taille et indexées par des familles de polynômes $\bold{x, y}$, alors chacune des familles $\bold{x, y}$ est une base de $A$.
\end{prop}

\begin{rem}
Nous insistons sur le fait que la proposition précédente est valable uniquement dans l'hypothèse où l'idéal est zéro-dimensionnel. Lors de nos expériences nous avons pu observer que dans le cas contraire, il est possible d'obtenir à la fin du processus de réduction des matrices $B(1)$ de tailles différentes suivant que l'on utilise les relations en $x$ ou en $y$ lors du processus de réduction. Le cas échéant, cette différence de taille finale est un phénomène qui reste à éclaircir.
\end{rem}

\subsubsection{Formules de Barnett et matrices compagnon}
Reprenons l'exemple précédent et définissons les matrices $X_1, X_2$ par les quotients
\begin{equation}
	X_1 = B(x_1)B(1)^{-1} =
	\begin{bmatrix}
		0 & -1 & 0\\
		-1 & 0 & -1\\
		-1 & 0 & 0
	\end{bmatrix},\quad
	X_2 = B(x_2)B(1)^{-1} =
	\begin{bmatrix}
		0 & 0 & 1\\
		1 & 0 & 0\\
		0 & 1 & -1
	\end{bmatrix}
\end{equation}
On vérifie que $X_1, X_2$ sont les matrices de multiplication par les variables $x_1, x_2$ dans la base $\bold{x}$ et sont donc les matrices compagnon associées à la base $\bold{x}$. D'une façon générale nous avons:
\begin{prop}
\label{Barnett_multi}
Lorsque le processus de réduction est terminé et que les matrices de Bezout sont écrites dans des bases $\bold{x, y}$, alors les matrices compagnon $X_j$, c'est à dire les matrices de multiplication par $x_j$ dans la base $\bold{x}$, peuvent se calculer grâce aux {\bf formules de Barnett}
\begin{equation}
	X_j = B(x_j)B(1)^{-1}
\end{equation}
\end{prop}

\begin{rem}
Comme dans le cas univariable nous avons pour tout $j=1,\cdots,n$,\\
$B(x_j)^{T}B(1)^{-T}$ est la matrice de multiplication par $y_j$ dans la base $\bold{y}$ \\
$B(1)^{-1}B(x_j)$ est la matrice de multiplication par $x_j$ dans la base $\hat{\bold{x}}$ \\
$B(1)^{-T}B(x_j)^{T}$ est la matrice de multiplication par $y_j$ dans la base $\hat{\bold{y}}$
\end{rem}

\subsubsection{Calcul numérique des racines}
Comme dans le cas univariable (voir Proposition \ref{compan2roots}) les racines du système polynomial $f_1, \cdots, f_n$ s'obtiennent numériquement en calculant les valeurs propres des matrices compagnons (\cite{AS}). Dans cet exemple les matrices $X_1, X_2$ fournissent les valeurs propres
$$
\begin{array}{c|c}
	x_1 & x_2 \\
	\hline
	-1.32472  & 0.75488 \\
	0.66236 + 0.56228i & -0.87744 + 0.74486i \\
	0.66236 - 0.56228i & -0.87744 - 0.74486i
\end{array}
$$
Puisque l'algèbre $A$ est commutative, les matrices $X_1, X_2$ commutent et ont donc les mêmes vecteurs propres. Lors du calcul il faut donc faire attention d'ordonner les valeurs propres pour qu'elles correspondent aux mêmes vecteurs propres. Dans l'exemple précédent on vérifie facilement que les couples $(x_1, x_2)$ ci-dessus sont bien des approximations des racines du système $f_1 = x_1^2 + x_1x_2^2 - 1, f_2 = x_1^2x_2 + x_1$.

\subsection{Structure bloc-triangulaire et rang numérique de $B(1)$}

\begin{figure}[h]
  \caption{Matrice $B(1)$, sparsité}
  \label{fig:B1_sparsity}
  \includegraphics[width=15cm, height=8cm]{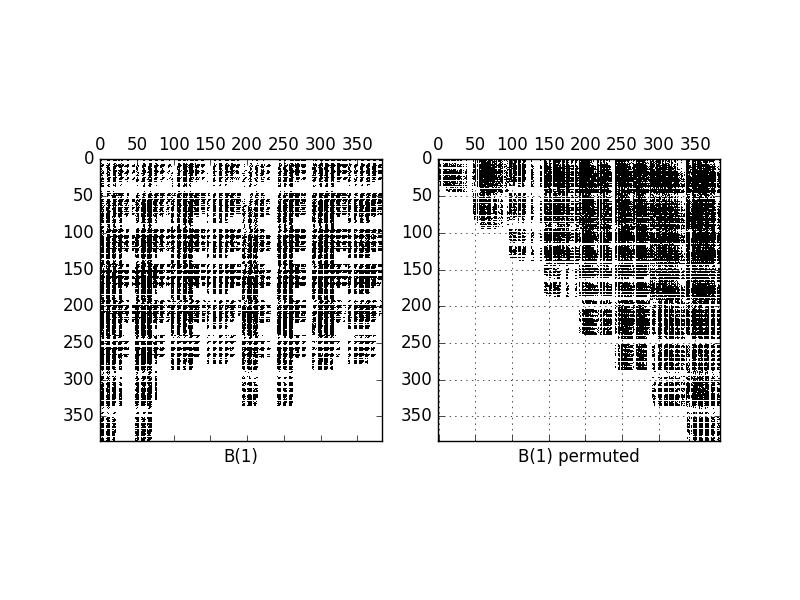}
\end{figure}
Dans le processus de réduction vu à la section \ref{sec:reduction_process} la première étape consiste à calculer le noyau de $B(1)$. Lorsque les coefficients des polynômes d'entrée sont entiers ou rationnels, ceci peut se faire de manière exacte au moyen d'un programme de calcul symbolique. La taille des entiers peut alors croître considérablement au cours des calculs et augmenter en conséquence le temps total de calcul et les besoins en mémoire du calculateur. Si par contre on veut effectuer l'ensemble des calculs en nombres flottants, ou si les coefficients d'entrée sont eux mêmes donnés sous forme numérique, alors on doit faire un calcul numérique du noyau. La méthode éprouvée pour cela, implémentée dans des packages d'algèbre linéaire numérique comme Matlab/Octave, Numpy ou Julia, est d'effectuer une factorisation QR ``rank revealing'' de $B(1)$, que nous appellerons factorisation QRP, c'est-à-dire accompagnée de pivots sur les colonnes. L'expérience montre que cette approche est souvent efficace mais peut s'avérer délicate à mettre en oeuvre si la taille de la matrice augmente. Montrons le sur un exemple. Nous choisissons $n =\input{n.tex}$ et un système polynomial $f$ de multidegré $\input{deg.tex}$. Seuls une quinzaine de monômes sont retenus pour chaque polynôme. Les coefficients, entiers, sont choisis aléatoirement entre $-t$ et $t$ avec, ici, $t=3$. Choisir une plus grande valeur de $t$ ne poserait aucun problème particulier si on utilise les matrices de Bezout mais on constate que le temps de calcul est excessivement long lorsqu'on utilise les bases de Grobner (voir Table \ref{tab:timings}).
Voici la liste des polynômes composant le système $f$ :\\
$f = \input{P.tex}$ \\
 La matrice de Bezout $B(1)$ est de taille $\input{Dx.tex}$ et possède une certaine structure, comme le montre la Figure \ref{fig:B1_sparsity}. Comme il parait difficile d'exploiter cette structure pour le calcul numérique du rang de la matrice $B(1)$, on doit recourir à une méthode numérique générale, par exemple une factorisation SVD ou une factorisation QRP ``rank revealing''. Choisissons cette deuxième méthode. Les termes diagonaux du facteur triangulaire $R$ sont triés en ordre décroissant, comme le montre la figure \ref{fig:B1_diag}.
\begin{figure}[h]
  \caption{Factorisation QRP de B(1), termes diagonaux}
  \label{fig:B1_diag}
  \includegraphics[width=15cm, height=8cm]{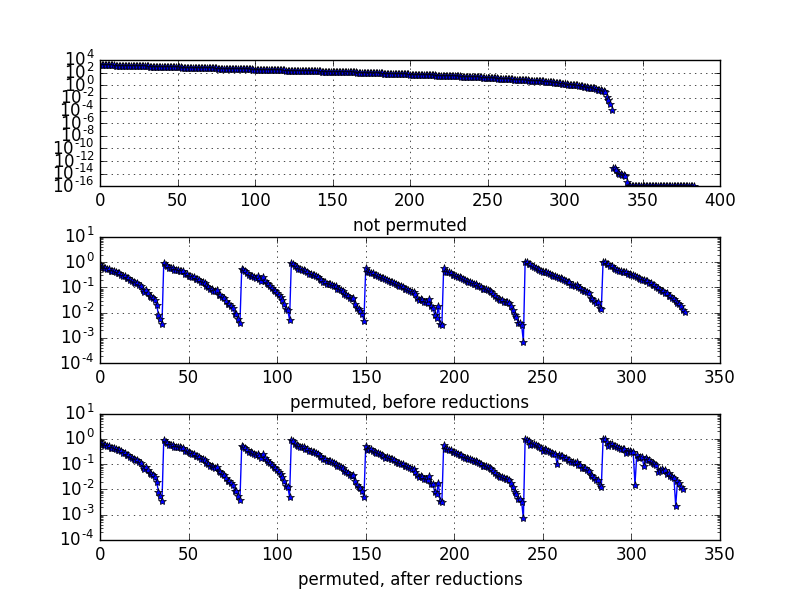}
\end{figure}
On s'aperçoit que les derniers termes non nuls décroissent vite, et qu'il peut devenir difficile de choisir un seuil au dessus duquel les termes diagonaux seront déclarés ``non nuls''. Les termes non nuls s'étendent de $10^4$ à $10^{-4}$. Le saut entre termes ``non-nuls'' et termes proches du epsilon machine a tendance à diminuer à mesure que la taille de la matrice augmente, ce qui rend le calcul du rang numérique difficile. Nous pouvons cependant améliorer, dans une certaine mesure, la situation précédente en exploitant une propriété de $B(1)$. En effet, en permutant lignes et colones de cette matrice d'une certaine façon, on peut arriver à une structure bloc-triangulaire de $B(1)$ (Figure \ref{fig:B1_diag}, subplot $2$). En appliquant à la matrice une factorisation QRP bloc après bloc, les termes diagonaux vont alors décroitre uniquement à l'intérieur de chaque bloc. La figure \ref{fig:B1_diag}, subplot $2$, montre la nouvelle disposition des termes diagonaux à la fin de la factorisation QRP, en traitant les blocs l'un après l'autre. Ici, les termes non nuls s'étendent de $10^0$ à $10^{-3}$. Le calcul du rang numérique est facilité et l'on trouve ici un rang égal à $\input{dim0.tex}$, qui correspond au nombre de termes dans la figure \ref{fig:B1_diag}, subplot $2$. Enfin la figure \ref{fig:B1_diag}, subplot $3$, montre la distribution des termes diagonaux dans la matrice finale, une fois les réductions faites, comme expliqué dans la section \ref{sec:reduction_process}. On voit que la distribution des termes est très proche de celle précédent les réductions. Le rang de la nouvelle matrice est $\input{dim.tex}$, c'est la dimension du quotient $A$, d'après la proposition~\ref{conjecture}.

\pagebreak

\subsection{Calcul numérique des racines}

\begin{figure}[h]
    \caption{Processus de réduction exécuté en arithmétique exacte}
  \label{fig:roots}
  \includegraphics[width=15cm, height=6cm]{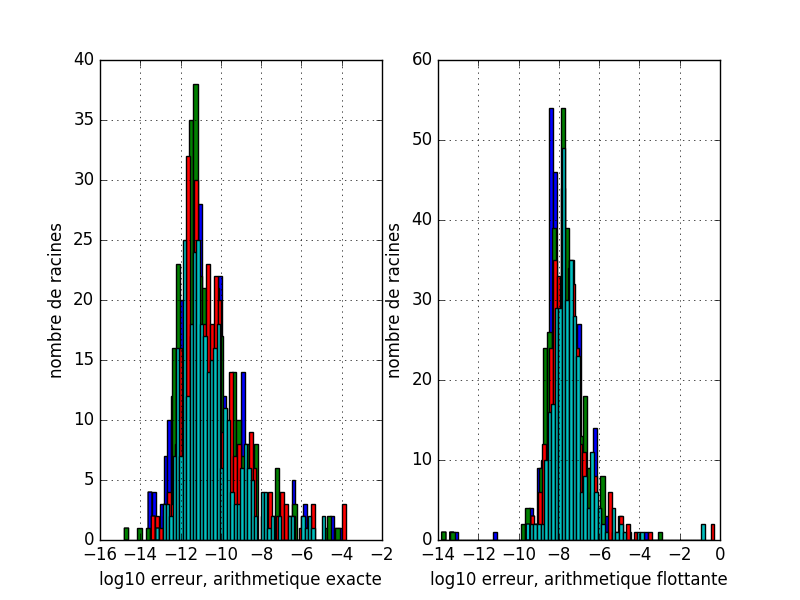}
\end{figure}
 Nous reprenons l'exemple ci-dessus. La matrice de Bezout $B(1)$, à coefficients entiers, est de taille \input{Dx.tex}. Après réductions on trouve que la dimension du quotient $A$ est \input{dim.tex}. En calculant numériquement les valeurs propres des matrices compagnon $X_j = B(x_j)B(1)^{-1}$ on obtient les racines du système polynomial $f$. On vérifie la qualité de chacune des racines obtenues en lui appliquant les polynômes $f_i, i=1,\cdots,n$. Les résultats sont représentés sous forme d'histogramme (Figure \ref{fig:roots}) o\`u le logarithme décimal de l'erreur est porté en abscisse. Sur le plot de gauche le processus de réduction est effectué en arithmétique exacte, sur le plot de droite il est effectué en arithmétique flottante. On constate (Table \ref{tab:timings}) que le temps de calcul en arithmétique flottante est plus court mais au prix d'une dégradation sensible de la qualité des résultats. On peut noter aussi que le calcul de la dimension du quotient, effectué par la méthode des bases de Grobner (fonction vector\_space\_dimension() de Sage), demande un temps beaucoup
plus long que lorsqu'on utilise les matrices de Bezout. Il semble de plus que ce temps de calcul (bases de Grobner) augmente considérablement avec la taille des coefficients entiers du système polynomial, ce qui explique notre choix de restreindre ces coefficients entre $t=-3$ et $t=3$ dans notre expérience.

\begin{table}[h]
    \caption{timings}
\label{tab:timings}
\begin{tabular}{llllr}
  Arithmétique & Méthode & Processus & Software & Timing \\ \hline
  \multirow{4}{*}{flottante} & \multirow{4}{*}{Bezout} & Construction matrices de Bezout & NumPy & $\input{construction_B_time.tex}$ ms \\ 
  & & Noyau de $B(1)$ & Octave & $\input{octave_triang_time.tex}$ ms \\
  & & Réduction matrices & Octave & $\input{octave_reduct_time.tex}$ ms \\ 
  & & Valeurs propres & SciPy & $\input{eigenstructure_time.tex}$ ms \\ \hline \hline
  \multirow{2}{*}{exacte} & Bezout & Réduction matrices & Sage & $\input{sage_reduct_time.tex}$ ms \\ 
  & Grobner & Vérification dimension Algèbre & Sage & $\input{sage_dimension_time.tex}$ ms
\end{tabular}
\end{table}

\section{Conclusion et perspectives}
Nous avons proposé une méthode de résolution numérique des systèmes polynômiaux en intersection complète. Cette méthode utilise exclusivement des techniques d'algèbre linéaire numérique. Le principe de la méthode est de nature algébrique mais fournit des racines dont on peut tester facilement la qualité numérique.

\end{document}

%% file: n.tex
  4

%% file: deg.tex
[2 ,2 ,2 ,2]

%% file: P.tex
[3x_0^2x_1^2x_2^2x_3^2 - x_0x_1^2x_2^2x_3^2 - 2x_0^2x_1^2x_2^2 - 3x_0^2x_1x_2x_3^2 + 3x_1^2x_2^2x_3^2 - x_0^2x_1x_2x_3 + 2x_0x_1x_2^2x_3 + x_1^2x_2^2x_3 + 2x_0^2x_2x_3^2 - 2x_0^2x_1x_2 + 3x_1x_2x_3^2 + x_2^2x_3^2 - 2x_3,\\ -2x_0^2x_1^2x_2^2x_3^2 + 3x_0^2x_1x_2^2x_3 - x_0^2x_1x_2x_3^2 + 2x_0x_1x_2^2x_3^2 + x_0^2x_1^2x_2 - x_0x_1^2x_2^2 - x_0x_2^2x_3^2 - 2x_0^2x_1^2 + x_0^2x_1x_3 - 2x_1x_2^2x_3 + 3x_1x_2x_3^2 - 3x_0x_1^2 + x_0x_2^2 - 3x_1x_2x_3 + 3x_0x_3^2 - 2x_0x_1,\\ -3x_0^2x_1^2x_2^2x_3^2 + 2x_0^2x_1x_2^2x_3^2 + 2x_0^2x_1^2x_2^2 - x_0x_1x_2^2x_3^2 + 2x_1^2x_2^2x_3^2 - x_0^2x_1x_2^2 - 3x_0^2x_1^2x_3 + 2x_0^2x_1x_3^2 - 2x_0x_1x_2x_3^2 + 2x_0^2x_1^2 - 2x_1x_2^2x_3 - 3x_0x_1^2 - 3x_1^2x_2 - x_0^2x_3 - x_2^2x_3 - 2x_2x_3^2 - 3x_0^2 - 3x_0x_2 + 2x_0x_3 - x_1x_3,\\ 3x_0^2x_1^2x_2 - x_0x_1^2x_2^2 + 2x_0^2x_1^2x_3 + 2x_1^2x_2^2x_3 - 2x_0x_2^2x_3^2 - 2x_1x_2^2x_3^2 - 3x_0^2x_1^2 + x_1x_2^2x_3 + 3x_0^2x_3^2 - x_0x_2x_3^2 + x_2^2x_3^2 + x_0^2x_1 + 2x_0x_1x_2 - x_1^2x_2 - 2x_0x_2^2 - x_0x_1x_3 + 3x_1x_3^2 - 2x_2 - 3]

%% file: Dx.tex
384

%% file: dim0.tex
331

%% file: dim.tex
330

%% file: construction_B_time.tex
2090

%% file: octave_triang_time.tex
117

%% file: octave_reduct_time.tex
190

%% file: eigenstructure_time.tex
2614

%% file: sage_reduct_time.tex
6663

%% file: sage_dimension_time.tex
827933